\documentclass[11pt]{article}

\usepackage[margin=.1in]{geometry} 
\usepackage{amsmath,amsthm,amssymb, mathrsfs}
\usepackage{graphicx}	
\usepackage{subfigure}
\usepackage{float}
\usepackage{fullpage}
\usepackage{listings}
\usepackage{amsfonts, dsfont}
\usepackage{enumerate}
\usepackage{xcolor}
\usepackage{graphicx}
\usepackage[mathscr]{euscript}
\usepackage[final]{pdfpages}
\usepackage{titlesec}
\usepackage{epstopdf}
\usepackage{pstricks}
\usepackage{pdfpages}
\usepackage{url}

\titleformat{\section}[hang]{}{\quad\large{\bf \thesection. }}{0in}{}{}
\titleformat{\subsection}[hang]{}{\quad {\bf \thesubsection. }}{0in}{}{}

\newcommand{\mcB}{\mathcal{B}}

\newcommand{\mcH}{\mathcal{H}}

\newcommand{\mcL}{\mathcal{L}}     
   
\newcommand{\mcN}{\mathcal{N}}

\newcommand{\mcS}{\mathcal{S}}

\newcommand{\N}{\mathbb{N}}

\renewcommand{\epsilon}{\varepsilon}

 \newtheorem{theorem}{Theorem}
 \numberwithin{theorem}{section}
 \numberwithin{theorem}{subsection}
 
 \newtheorem{lemma}{Lemma}
 \numberwithin{lemma}{section}
 \renewcommand*{\thetheorem}{%
   \ifnum\value{subsection}=0 %
     \thesection
   \else
     \thesubsection
   \fi
   .\arabic{theorem}%
 }

 \graphicspath{{/Users/CharlesBabbage/Documents/School/UTA/Research/Mywork/ModelGeneralizations/CorrelatedNodes/},{/Users/CharlesBabbage/Documents/School/UTA/Research/Mywork/MyPapers/}}
 
\begin{document}

\noindent\parbox[c]{4in}{\large\textsf{\bf RELIABILITY OF SERIES AND PARALLEL SYSTEMS WITH CORRELATED NODES}}
 
 \noindent 
 
 \begin{flushleft}
 {\bf Rachel Traylor$^{1}$} \\
{\em $^{1}$ University of Texas at Arlington, Arlington, Texas, United States} 
 \end{flushleft}
 
{\em $\qedsymbol$ Consider a system of $N$ components in which traffic arrives as separate but correlated nonhomogenous Poisson streams to each node rather than passing into the system at one entry point. A method is given to construct such systems mathematically and derive the survival function of the system for any logical topology. } \vspace{.25in}\\


{\bf Mathematics Subject Classification} Classify here
 \section{\large{ \bf INTRODUCTION}}
 \label{introduction to correlated nodes section}

    \begin{figure}[H]
	\centering
     \includegraphics[scale=0.5]{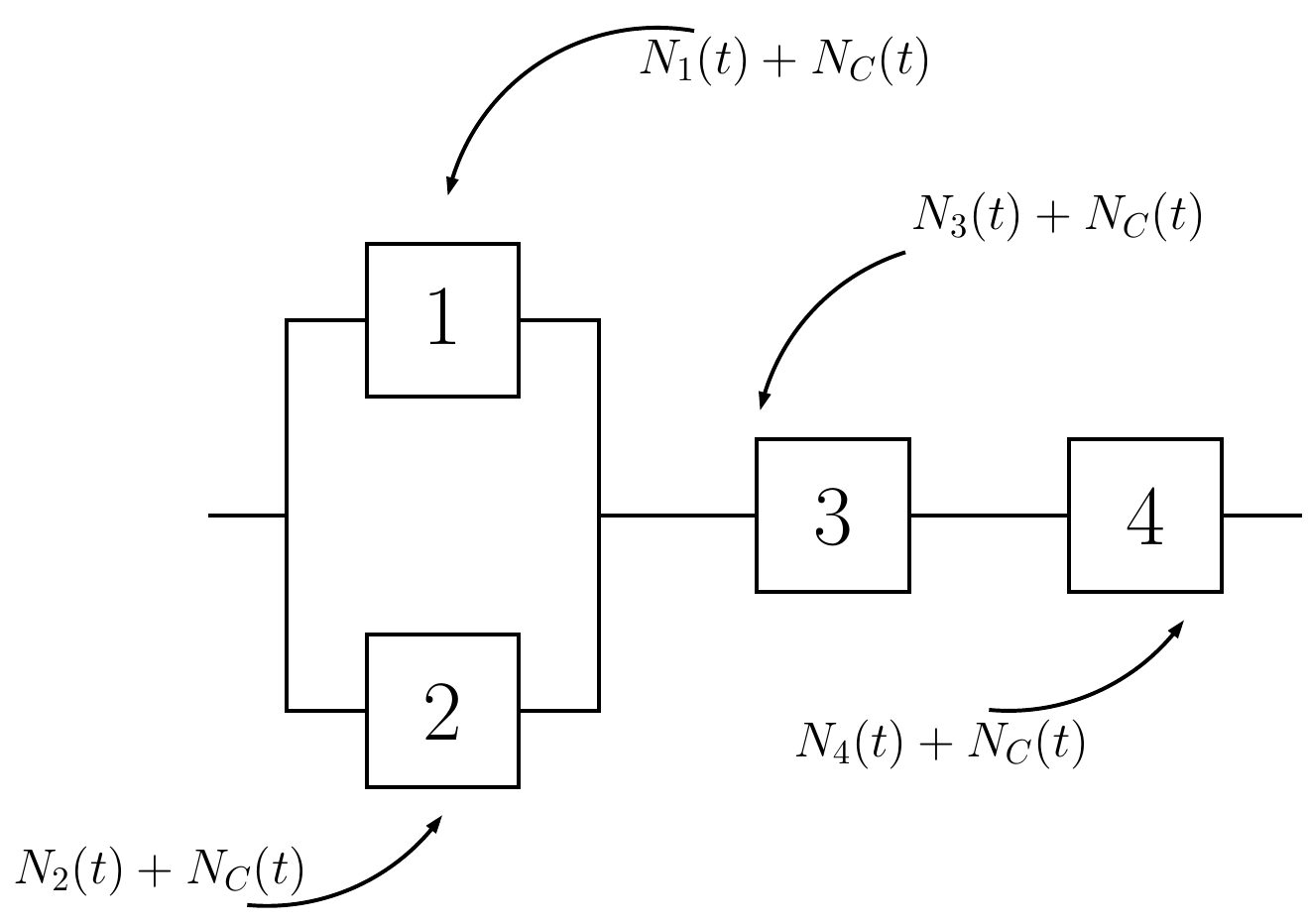} 
     \caption{Logical Topology for a Hypothetical System with Correlated Traffic Streams}
     \label{Beanstalk topology}
	 \end{figure}

\hspace{.6cm}Consider a system with multiple nodes such that the external traffic goes to each node individually. We may view many systems in retail, manufacturing, and IT this way from the perspective of system reliability. For example, instead of viewing a series of checkout registers as a $G/G/K$ queueing system, we may instead view the system as having a parallel logical topology of $K$ nodes with separate but correlated traffic streams to each node. A manufacturing system in which each manufacturing site is responsible for assembling a portion of a widget may have separate shipments of raw materials arriving to each location. \par
In each of these examples, the separate traffic streams are certainly correlated. We will model the arrivals to each node using a nonhomogeneous Poisson process, and introduce a \textit{correlator process} that will ensure all nodes are correlated but conditionally independent. \par

Let 
	$\{N_{1}(t): t \geq 0\}, \{N_{2}(t): t \geq 0\},\ldots, \{N_{K}(t): t \geq 0\}$
and a \textit{correlator process}
	$\{N_{c}(t): t \geq 0\}$ 
be mutually independent nonhomogenous Poisson processes (NHPPs) with intensities 
	$\lambda_{i}(t), i = 1,...,K$ 
and 
	$\lambda_{c}(t)$,
respectively. Now suppose there are $K$ components (or queues) in this system, denoted
	$Q_{\ell,c},  \ell = 1,...,K$
 such that the arrival processes 
    $\{\mcN_{\ell,c}(t): t \geq 0\}, \ell = 1,...,K$
are given by 
$\mcN_{\ell,c}(t) = N_{\ell}(t)+ N_{c}(t)$. \par 
By Lemma~\ref{lemma: sum of NHPP is a NHPP}, the sum of $n$ independent NHPPs remains a NHPP.
Thus, since $\{N_{\ell}\}_{\ell=1}^{K}$ and $\{N_{c}(t)\}$ are mutually independent, and 
$\mcN_{\ell}(t) = N_{\ell}(t) + N_{c}(t)$ is a NHPP, 
$E[\mcN_{\ell,c}(t)] = \lambda_{\ell}(t)+\lambda_{c}(t)$. \par
The covariance of $\mcN_{i}$ and $\mcN_{j}$ is given by 
\[Cov(\mcN_{i,c}, \mcN_{i,c})(t) = E[N_{q_{i}}N_{q_{j}}] - E[N_{q_{i}}]E[N_{q_{j}}] = \lambda_{c}(t) \nonumber			
\]
and the correlation between $\mcN_{i,c}$ and $\mcN_{j,c}$ is thus given by 
\[\rho_{\mcN_{i,c}, \mcN_{j,c}(t)} = \frac{Cov(\mcN_{i,c}, \mcN_{j,c})}{\sigma_{\mcN_{i,c}}\sigma_{\mcN_{j,c}}} = \frac{\lambda_{c}(t)}{\sqrt{(\lambda_{i}(t)+\lambda_{c}(t))(\lambda_{j}(t)+\lambda_{c}(t))}}\]

The rest of the paper is organized as follows. In Section~\ref{sec: 2 component section} we derive the survival function for a two compoment series and parallel system with correlated nodes to show the full derivation and calculation. Section~\ref{sec: K component systems} generalizes Section~\ref{sec: 2 component section}, and Section~\ref{sec: other architectures} details different examples and gives the general method for derivation of any logical system architecture.

\section{\large{ \bf SURVIVAL FUNCTION OF SYSTEM WITH TWO CORRELATED NODES}}
\label{sec: 2 component section}

In~\cite{traylor}, Traylor derives the unconditional survival function for a single server, or node under the following conditions:
\begin{enumerate}[(1)]
\item Service requests and/or shipments, henceforth known as jobs, arrive to the node via a NHPP  $\{N(t):t \geq 0\}$ with intensity function $\lambda(t)$. The arrival times are denoted $\{T_{j}\}_{j=1}^{N(t)}$.
\item Each job $j$ adds a random stress $\mcH_{j}$ to the system, increasing the breakdown rate until completion. The stresses $\{\mcH_{j}\}_{j=1}^{N(t)} \overset{i.i.d.}\sim \mcH$, where WLOG $\mcH$ is a discrete random variable with finite sample space $\{\eta_{i},...,\eta_{m}\}$ and pmf $P(\mcH = \eta_{i}) = p_{i}, i = 1,...,m$. The stresses are mutually independent of all arrival times $\{T_{j}\}_{j=1}^{N(t)}$. 
\item In addition, the jobs only add stress until completion of service.
\item The service times of each job are denoted $\{W_{j}\}_{j=1}^{N(t)} \overset{i.i.d.}\sim G_{W}(w)$ with pdf $g_{W}(w) = \frac{dG_{W}(w)}{dw}$ and are mutually independent of all stresses $\{\mcH_{j}\}_{j=1}^{N(t)}$ and service times $\{W_{j}\}_{j=1}^{N(t)}$
\end{enumerate}
We now extend~\cite{traylor} and~\cite{chalee} and derive the unconditional survival function for the following systems.

\subsection{\textbf{Series System}}
\label{2 component series}

Suppose two nodes $Q_{\ell,1}$ and $Q_{\ell,2}$ are arranged in series with arrival processes $\mcN_{1}(t) = N_{1}(t)+N_{c}(t)$ and $\mcN_{2}(t) = N_{2}(t) + N_{c}(t)$.
Let the NHPPs $\{N_{\ell}: t \geq 0\}, \ell = 1,2 \text{ and } c$ have arrival times $\{T_{j_{\ell}}\}_{j_{\ell} = 1}^{N_{\ell}(t)}$, service times $\{W_{j_{\ell}}\}_{j_{\ell}=1}^{N_{\ell}(t)}$, and stresses $\{\mcH_{j_{\ell}}\}_{j_{\ell} = 1}^{N_{\ell}(t)}$  as in~\cite{traylor}.  Assume that all stresses $\mcH_{j_{\ell}} \overset{i.i.d.}{\sim} \mcH$. In addition,  all service times regardless of node are i.i.d. with distribution $G_{W}(w)$ and pdf $g_{W}(w)$. Let the baseline breakdown rate for $Q_{\ell,c}$ be $r_{0_{\ell,c}}(t), \ell = 1,2$. Jobs in both queues add stress to the server until completion. Then for $Q_{\ell,c}$, the breakdown rate process for each node is given by 

\begin{equation}
\mcB_{\ell,c}(t) = r_{0_{\ell,c}}(t) + \sum_{j_{\ell} = 1}^{N_{\ell}(t)}\mcH_{j_{\ell}}\mathds{1}(T_{j_{\ell}} \leq t \leq T_{j_{\ell}} + W_{j_{\ell}}) + \sum_{j_{c} = 1}^{N_{c}(t)}\mcH_{j_{c}}\mathds{1}(T_{j_{c}} \leq t \leq T_{j_{c}} + W_{j_{c}})
\label{Ql_breakdownrate_process}
\end{equation}

The system survives past time $t$ if and only if both $Q_{1,c}$ and $Q_{2,c}$ survive past time $t$. Let 
	$Y_{i,c}, i = 1,2 $ 
	be the random length of the node lifetime under workload (or renewal cycle if the node can be rebooted) 
	$Q_{i,c}, i = 1,2$, 
	and 
	$Y_{S}$ 
	the system life under workload. \par 

$Q_{1,c}$ and $Q_{2,c}$ are conditionally independent under $\{N_{c}(t): t \geq 0\}, \{T_{j_{c}} = t_{j_{c}}\}_{j_{c}=1}^{N_{c}(t)}$, $\{W_{j_{c}} = w_{j_{c}}\}_{j_{c}=1}^{N_{c}(t)}$, and 
 $\{\mcH_{j_{c}} = \eta_{i_{j_{c}}}\}_{j_{c}=1}^{N_{c}(t)}$. Thus

\begin{align}
P\left(Y_{S}> t \right.&\left| N_{c}(t), \{T_{j_{c}}\}_{j_{c}=1}^{N_{c}(t)}, \{W_{j_{c}}\}_{j_{c}=1}^{N_{c}(t)}, \{\mcH_{j_{c}}\}_{j_{c}=1}^{N_{c}(t)} \right.\left.\right) \nonumber \\
		 &= P\left(Y_{1,c} > t \cap Y_{2,c} > t \left| N_{c}(t), \{T_{j_{c}}\}_{j_{c}=1}^{N_{c}(t)}, \{W_{j_{c}}\}_{j_{c}=1}^{N_{c}(t)}, \{\mcH_{j_{c}}\}_{j_{c}=1}^{N_{c}(t)}  \right.\right) \nonumber \\
		 &=\prod_{i=1}^{2} P\left(Y_{i}> t\left| N_{c}(t), \{T_{j_{c}}\}_{j_{c}=1}^{N_{c}(t)}, \{W_{j_{c}}\}_{j_{c}=1}^{N_{c}(t)}, \{\mcH_{j_{c}}\}_{j_{c}=1}^{N_{c}(t)}  \right.\right)
\label{conditional_survival_series_system_2_nodes}
\end{align}
 
By Lemma~\ref{lemma: conditional survival function for a node}, 
  \begin{align}
  P&\left(\right.Y_{\ell,c} > t \left| N_{c}(t), \{T_{j_{c}}\}_{j_{c}=1}^{N_{c}(t)}, \{W_{j_{c}}\}_{j_{c}=1}^{N_{c}(t)}, \{\mcH_{j_{c}}\}_{j_{c}=1}^{N_{c}(t)}  \right.\left.\right) \nonumber \\
  	&= \bar{F}_{0_{\ell,c}}(t)\exp\left(- \sum_{j_{c}=1}^{N_{c}(t)}\mcH_{j_{c}}\min(W_{j_{c}}, t-T_{j_{c}})\right)\exp\left(-E_{\mcH}\left[\mcH\int_{0}^{t}\exp(-\mcH w)m_{\ell}(t-w)\bar{G}_{W}(w)dw\right]\right)
  	\label{conditional survival function for a node equation}
  \end{align}
  where $\bar{F}_{0_{\ell,c}} = \exp\left(-\int_{0}^{t}r_{0}(x)dx\right)$.
  Thus from~\eqref{conditional_survival_series_system_2_nodes}, 
 \begin{align}
 P\left(Y_{s}> t \right.&\left| N_{c}(t), \{T_{j_{c}}\}_{j_{c}=1}^{N_{c}(t)}, \{W_{j_{c}}\}_{j_{c}=1}^{N_{c}(t)}, \{\mcH_{j_{c}}\}_{j_{c}=1}^{N_{c}(t)} \right.\left.\right) \nonumber \\
 &= \bar{F}_{0_{1,c}}(t)\bar{F}_{0_{2,c}}(t)\exp\left(- 2\sum_{j_{c}=1}^{N_{c}(t)}\mcH_{j_{c}}\min(W_{j_{c}}, t-T_{j_{c}})\right)\nonumber \\
 &\quad\times \exp\left(-E_{\mcH}\left[\mcH\int_{0}^{t}(m_{1}(t-w)+m_{2}(t-w))\exp(-\mcH w)\bar{G}_{W}(w)dw\right]\right)
 \label{conditional 2 comp series}
 \end{align}
 
The unconditional survival function for the two-component correlated series system is given in the following theorem.
 
 \begin{theorem}
 Let $\{N_{1}(t): t \geq 0\}, \{N_{2}(t): t \geq 0\},$ and $\{N_{c}(t): t \geq 0\}$ be independent NHPPs with intensities $\lambda_{1}(t), \lambda_{2}(t),$ and $\lambda_{c}(t)$, respectively. Suppose all arrival times $\{T_{j_{\alpha}}\}_{j_{\alpha} = 1}^{N_{\alpha}(t)}, \alpha = 1,2; c$ are independent. Let all service times $\{W_{j_{\alpha}}\}_{j_{\alpha} = 1}^{N_{\alpha}(t)}, \alpha = 1,2; c$ be i.i.d. with pdf $g_{W}(w)$ and distribution $G_{w}(w)$ and mutually independent of all arrival times. Let all stresses  $\{\mcH_{j_{\alpha}}\}_{j_{\alpha} = 1}^{N_{\alpha}(t)}, \alpha = 1,2; c$ be i.i.d. with distribution $\mcH$ as given in Condition (2), Section~\ref{sec: 2 component section} and mutually independent of arrival times and service times. Suppose we have a system of two components ($Q_{1,c}$ and $Q_{2,c}$) arranged logically in series, where each component has a  arrival process $\mcN_{i}(t) = N_{i}(t) + N_{c}(t), i = 1,2$. Then the survival function of the system $S_{Y_{s}}(t)$ is given by
 \begin{align}
 S_{Y_{s}}(t) &= \bar{F}_{0_{1}}(t)\bar{F}_{0_{2}}(t)\exp\left(-E_{\mcH}\left[\mcH\int_{0}^{t}(m_{1}(t-w)+m_{2}(t-w))\exp(-\mcH w)\bar{G}_{W}(w)dw\right] \right)\nonumber \\
 	&\hspace*{4cm}\times\exp\left(-2E_{\mcH}\left[\mcH\int_{0}^{t}\exp(-2\mcH w)m_{c}(t-w)\bar{G}_{w}(w)dw\right]\right) 
 \end{align}
 \label{thm: survival_series_correlated}
 \end{theorem}
 
 \begin{proof}
As in the proof of  Theorem 3.1,~\cite{traylor},
\begin{align}
S_{Y_{s}}(t) &= E\left[E\left[ P\left(Y_{s}> t \right.\left| N_{z}(t), \{T_{j_{c}}\}_{j_{c}=1}^{N_{z}(t)}, \{W_{j_{c}}\}_{j_{c}=1}^{N_{z}(t)}, \{\mcH_{j_{c}}\}_{j_{c}=1}^{N_{z}(t)} \right.\left.\right)\left.\right]\right|N_{c}(t), \{\mcH_{j_{c}}\}_{j_{c}=1}^{N_{c}(t)}\right] \nonumber \\
&= \bar{F}_{0_{1}}(t)\bar{F}_{0_{2}}(t)\exp\left(-E_{\mcH}\left[\mcH\int_{0}^{t}(m_{1}(t-w)+m_{2}(t-w))\exp(-\mcH w)\bar{G}_{W}(w)dw\right] \right)\nonumber \\
	&\times E\left[E\left[\exp\left(- 2\sum_{j_{c}=1}^{N_{c}(t)}\mcH_{j_{c}}\min(W_{j_{c}}, t-T_{j_{c}})\left.\right)\right| N_{c}(t), \{\mcH_{j_{c}}\}_{j_{c}=1}^{N_{c}(t)}\right]\right] \nonumber 
\end{align}
We may calculate the given expectation in the exact same way as in the RSBR proof replacing $\mcH_{j_{c}}$ with $2\mcH_{j_{c}}$ to immediately see the given result.
 \end{proof}
 
\subsection{\textbf{Parallel System}}
 \label{subsec: 2 comp parallel}
 
 Now suppose we retain all the same conditions (1)-(4), but we change the component structure to a parallel system. In this case, the system fails only if both components fail. Thus, conditioning on the entire $\{N_{c}\}$ process, both $Q_{1,c}$ and $Q_{2,c}$ are now independent, and 
 
\begin{align*}
P\left(\right.Y_{s} < t &\left| N_{c}(t), \{T_{z_{j}}\}_{j_{c}=1}^{N_{c}(t)}, \{W_{j_{c}}\}_{j_{c}=1}^{N_{c}(t)}, \{\mcH_{j_{c}}\}_{j_{c}=1}^{N_{c}(t)}  \right.\left.\right) \\
 &= \prod_{\ell=1}^{2}P\left(\right.Y_{\ell,c} < t \left| N_{c}(t), \{T_{j_{c}}\}_{j_{c}=1}^{N_{c}(t)}, \{W_{j_{c}}\}_{j_{c}=1}^{N_{c}(t)}, \{\mcH_{j_{c}}\}_{j_{c}=1}^{N_{c}(t)}  \right.\left.\right) \nonumber
\end{align*}

Hence, the conditional survival function of the parallel system is given by 
\begin{align*}
P\left(\right.Y_{s} > t &\left| N_{c}(t), \{T_{j_{c}}\}_{j_{c}=1}^{N_{c}(t)}, \{W_{j_{c}}\}_{j_{c}=1}^{N_{c}(t)}, \{\mcH_{j_{c}}\}_{j_{c}=1}^{N_{c}(t)}  \right.\left.\right)   \\
&= 1-P\left(\right.Y_{s} < t \left| N_{c}(t), \{T_{z_{j}}\}_{j_{c}=1}^{N_{c}(t)}, \{W_{j_{c}}\}_{j_{c}=1}^{N_{c}(t)}, \{\mcH_{j_{c}}\}_{j_{c}=1}^{N_{c}(t)}  \right.\left.\right) \\ 
 &= \sum_{\ell=1}^{2}P\left(\right.Y_{\ell,c} > t \left| N_{c}(t), \{T_{j_{c}}\}_{j_{c}=1}^{N_{c}(t)}, \{W_{j_{c}}\}_{j_{c}=1}^{N_{c}(t)}, \{\mcH_{j_{c}}\}_{j_{c}=1}^{N_{c}(t)}  \right.\left.\right)  \\
 &\quad - \prod_{\ell=1}^{2}P\left(\right.Y_{\ell,c} > t \left| N_{c}(t), \{T_{j_{c}}\}_{j_{c}=1}^{N_{c}(t)}, \{W_{j_{c}}\}_{j_{c}=1}^{N_{c}(t)}, \{\mcH_{j_{c}}\}_{j_{c}=1}^{N_{c}(t)}  \right.\left.\right) 
\end{align*}

Using Lemma~\ref{lemma: conditional survival function for a node}, 

\begin{align}
P&\left(\right.Y_{s} > t \left| N_{c}(t), \{T_{j_{c}}\}_{j_{c}=1}^{N_{c}(t)}, \{W_{j_{c}}\}_{j_{c}=1}^{N_{c}(t)}, \{\mcH_{j_{c}}\}_{j_{c}=1}^{N_{c}(t)}  \right.\left.\right)  \nonumber \\
	&= \bar{F}_{0_{1}}(t)\exp\left(-\sum_{j_{c}=1}^{N_{c}(t)}\mcH_{j_{c}}\min(W_{j_{c}}, t-T_{j_{c}}\right)\exp\left(-E_{\mcH}[\mcH\int_{0}^{t}\exp(-\mcH w)m_{1}(t-w)\bar{G}_{W}(w)dw]\right) \nonumber \\
	&+\bar{F}_{0_{2}}(t)\exp\left(-\sum_{j_{c}=1}^{N_{c}(t)}\mcH_{j_{c}}\min(W_{j_{c}}, t-T_{j_{c}}\right)\exp\left(-E_{\mcH}[\mcH\int_{0}^{t}\exp(-\mcH w)m_{2}(t-w)\bar{G}_{W}(w)dw]\right) \nonumber \\
	&+\bar{F}_{0_{1}}(t)\bar{F}_{0_{2}}(t)\exp\left(-2\sum_{j_{c}=1}^{N_{c}(t)}\mcH_{j_{c}}
	\min(W_{j_{c}}, t-T_{j_{c}}\right)\nonumber \\
	&\hspace*{1.75in}\times\exp\left(-E_{\mcH}[\mcH\int_{0}^{t}(m_{1}(t-w)+m_{2}(t-w))\exp(-\mcH w)\bar{G}_{W}(w)dw]\right)
\end{align}

Now, we may find $S_{Y_{S}}(t)$ for the parallel system in the same manner as the series system. Then, denoting $f^{r,q}_{\mcH}(t) = q\mcH\int_{0}^{t}e^{-q\mcH w}m_{r}(t-w)\bar{G}_{w}(w)dw$, where $r = \{1,...,K,c\}$ and $q \in \N$.

\begin{align}
P(Y_{s}> t) &= \bar{F}_{0_{1}}(t)\exp\left(-E_{\mcH}[f^{1,1}_{\mcH}(t) + f^{c,1}_{\mcH}(t)]\right) + \bar{F}_{0_{2}}\exp\left(-E_{\mcH}[f^{2,1}_{\mcH}(t) + f^{c,1}_{\mcH}(t)]\right) \nonumber \\
	&\quad - \bar{F}_{0_{1}}(t)\bar{F}_{0_{2}}(t)\exp\left(-E_{\mcH}[f^{1,1}_{\mcH}(t) + f^{2,1}_{\mcH}(t)]\right)\exp\left(-E_{\mcH}\left[f^{c,2}_{\mcH}(t)\right]\right)
\end{align}

\section{\large{ \bf GENERALIZATION TO $\mathbf{K}$ COMPONENT SYSTEMS}}
\label{sec: K component systems}

The systems in Sections~\ref{2 component series} and \ref{subsec: 2 comp parallel} are generalized to $K$ components, all with the same correlator process $N_{c}$. Thus we now have components $Q_{1,c},...,Q_{K,c}$ with arrival processes $\mcN_{\ell,c}(t) = N_{\ell}(t) + N_{c}(t), \ell = 1,...,K$. 

\subsection{\textbf{Survival Function of Series System with $K$ correlated components}}
\label{subsec: K component series system}

The conditional survival function for a system with $K$ correlated nodes, all correlated by the same process $N_{c}(t)$ is a straightforward generalization of \eqref{conditional 2 comp series} and is given by 

 \begin{align}
 P\left(Y_{s}> t \right.&\left| N_{c}(t), \{T_{j_{c}}\}_{j_{c}=1}^{N_{c}(t)}, \{W_{j_{c}}\}_{j_{c}=1}^{N_{c}(t)}, \{\mcH_{j_{c}}\}_{j_{c}=1}^{N_{c}(t)} \right.\left.\right) \nonumber \\
 &= \left(\prod_{\ell = 1}^{K}\bar{F}_{\ell,c}\right)\exp\left(- K\sum_{j_{c}=1}^{N_{c}(t)}\mcH_{j_{c}}\min(W_{j_{c}}, t-T_{j_{c}})\right)\nonumber \\
 &\quad\times exp\left(-E_{\mcH}\left[\mcH\int_{0}^{t}\left(\sum_{\ell=1}^{K}m_{\ell}(t-w)\right)\exp(-\mcH w)\bar{G}_{W}(w)dw\right]\right)
 \label{conditional k comp series}
 \end{align}
 
 and thus the unconditional survival function for a series system of $K$ correlated components is given by 
 
 \begin{align}
  S_{Y_{s}}(t) &= \left(\prod_{\ell = 1}^{K}\bar{F}_{\ell,c}\right)\exp\left(-E_{\mcH}\left[\mcH\int_{0}^{t}\left(\sum_{\ell=1}^{K}m_{\ell}(t-w)\right)\exp(-\mcH w)\bar{G}_{W}(w)dw\right] \right)\nonumber \\
  	&\hspace*{4cm}\times\exp\left(-KE_{\mcH}\left[\mcH\int_{0}^{t}\exp(-K\mcH w)m_{c}(t-w)\bar{G}_{w}(w)dw\right]\right) 
  	  \label{survival_series_k_correlated}
  \end{align}
  
  \subsection{\textbf{Survival Function of Parallel System with $\mathbf{K}$ correlated components}}
  \label{subsec: K parallel components}

For a system with $K$ components in parallel, note again that the systems fails if and only if every component fails. Let $\xi_{\ell,c}(t) =  P\left(Y_{\ell,c}> t \right.\left| N_{c}(t), \{T_{j_{c}}\}_{j_{c}=1}^{N_{c}(t)}, \{W_{j_{c}}\}_{j_{c}=1}^{N_{c}(t)}, \{\mcH_{j_{c}}\}_{j_{c}=1}^{N_{c}(t)} \right.\left.\right)$. Then the conditional survival function for the $K$-component parallel system is given by 
\[P\left(Y_{S}> t \right.\left| N_{c}(t), \{T_{j_{c}}\}_{j_{c}=1}^{N_{c}(t)}, \{W_{j_{c}}\}_{j_{c}=1}^{N_{c}(t)}, \{\mcH_{j_{c}}\}_{j_{c}=1}^{N_{c}(t)} \right.\left.\right) = 1 - \prod_{\ell = 1}^{K}(1-\xi_{\ell,c}(t))\]
Denote $\mathbf{\xi} = (\xi_{1,c},...,\xi_{K,c})$, $\mathbf{1}$ as the $K-$tuple of 1's. Then
\begin{align}
1-\prod_{\ell=1}^{K}(1-\xi_{\ell,c}(t)) &= 1-\sum_{\mathbf{\nu} \leq \mathbf{1}}{\mathbf{1} \choose \mathbf{\nu}}(-1)^{\mathbf{1}-\mathbf{\nu}}\mathbf{\xi}^{\mathbf{1}-\mathbf{\nu}} \nonumber \\
		&= 1-\sum_{s_{1} = 0}^{1}\cdots\sum_{s_{K}=0}^{1}(-1)^{\scriptstyle 1- s_{1}}\xi_{\scriptscriptstyle 1,c}^{\scriptstyle 1-s_{\scriptscriptstyle 1}}\cdots(-1)^{\scriptstyle 1-s_{\scriptscriptstyle K}}\xi_{\scriptscriptstyle K,c}^{\scriptstyle 1-s_{\scriptscriptstyle K}}
\label{eq: k_parallel_cond}
\end{align}

Let $\mcS = \{ s = (s_{1},...,s_{K}): s_{\ell} = {0,1} \}$ denote all possible combinations of $s$ in the terms of~\eqref{eq: k_parallel_cond}.  Let $\mcL_{\sigma} = \{\ell : s_{\ell} = 0 \text{ in } s_{\sigma} , s_{\sigma} \in \mcS\}$. Then we may express~\eqref{eq: k_parallel_cond} in the following way:
\begin{equation}
1-\prod_{\ell=1}^{K}(1-\xi_{\ell,c}(t))= 1-\sum_{s \in \mcS}(-\mathbf{1})^{s}\xi^{s}
\end{equation} 
Using Lemma~\ref{lemma: conditional survival function for a node},

\begin{align}
\xi^{s} = \left(\prod_{\ell \in \mcL_{\sigma}}\bar{F}_{0_{\ell,c}}(t)\right)&\exp\left(-E_{\mcH}\left[\mcH\int_{0}^{t}\left(\sum_{\mcL_{\sigma}}m_{\ell(t-w)}\right)e^{-\mcH w}\bar{G}_{W}(w)dw\right]\right)\nonumber \\
&\hspace*{2in}\times\exp\left(-|\mcL_{\sigma}|\sum_{j_{c}=1}^{N_{c}(t)}\mcH_{j_{c}}\min(W_{j_{c}}, t-T_{j_{c}})\right)
\label{eq: parallel_k_prod_terms}
\end{align}

Thus, in a similar fashion to Theorem~\ref{thm: survival_series_correlated}, the survival function for a system of $K$ correlated components in parallel is given by

\begin{align}
S_{Y_{s}}(t) &= 1-\sum_{s \in \mcS}(-1)^{s}E[\xi^{s}] \nonumber \\
				&= 1-\sum_{s \in \mcS}(-1)^{s}\left(\prod_{\ell \in \mcL_{\sigma}}\bar{F}_{0_{\ell,c}}(t)\right)\exp\left(-E_{\mcH}\left[\int_{0}^{t}\bar{G}_{W}(w)\left(\mcH e^{-\mcH w}\sum_{\ell \in \mcL_{\sigma}}\left[m_{\ell}(t-w)\right] \right.\right.\right. \nonumber \\
				&\hspace*{3.5in}+ \left.\left.\left.|\mcL_{\sigma}|\mcH e^{-|\mcL_{\sigma}|\mcH w}m_{c}(t-w)\right)\right]\right)
\end{align}

\section{\large{ \bf SELECTED OTHER LOGICAL SYSTEM ARCHITECTURES AND A GENERALIZED METHOD FOR OBTAINING SYSTEM SURVIVAL FUNCTIONS}}
\label{sec: other architectures}

\hspace{.6cm}This section extends the same principle of multiple nodes with one correlator process to other selected logical system architectures. It is common in the analysis of system reliability to employ \textit{structure functions} which define the system state as a function of the component states. These structure functions give the system state (with binary assumption of \textit{working} or \textit{failed}) as a function of component states~\cite{reliability}. Denote $x_{i}$ as the state of component $i$. Then
\[x_{i} := \begin{cases}
0,& \text{ if $i$ has failed} \\
1, & \text{ if $i$ is working}
\end{cases}\]

Then the structure function of a system with $n$ components is given by 
\[\phi(\mathbf{x}) = \begin{cases}
0, &\text{ if the system has failed when in state $\mathbf{x}$} \\
1, & \text{ if the system is working when in state $\mathbf{x}$}
\end{cases}\]
where $\mathbf{x} = (x_{1},...,x_{n})$  is the state vector. \par 
As a brief example, the structure function of the series system with $K$ components is given by 
$\phi_{\text{series}}(\mathbf{x}) = \prod_{\ell=1}^{K}x_{\ell}$. Replacing the binary $x_{\ell}$ by the conditional survival function 
\[P\left(Y_{\ell,c} > t \left| N_{c}(t), \{T_{j_{c}}\}_{j_{c}=1}^{N_{c}(t)}, \{W_{j_{c}}\}_{j_{c}=1}^{N_{c}(t)}, \{\mcH_{j_{c}}\}_{j_{c}=1}^{N_{c}(t)}\right.\right)\] 
for each component given by Lemma~\ref{lemma: conditional survival function for a node}, the conditional survival function for the series system given in ~\eqref{conditional k comp series} is completely analogous to the structure function $\phi$.  \par 

The structure function for a parallel system is given by 
$\phi_{\text{parallel}}(\mathbf{x}) = 1-\prod_{\ell=1}^{K}(1-x_{\ell})$
which may be expanded into the form of ~\eqref{eq: k_parallel_cond}. Thus, using similar logic, the conditional survival function of a parallel system is analogous to the structure function of a parallel system. Therefore, for both a series and parallel system, the binary state variable $x_{\ell}$ and the conditional survival function for node $\ell$ may be viewed to be in a one-to-one correspondence of sorts. Thus, the system survival function is isomorphic to the system structure function for both series and parallel systems. Since every logical system architecture can be written either as a series system comprised of parallel subsystems or a parallel system comprised of series subsystems, we only need the structure function expressed as a linear combination of powers of $x_{\ell}: \ell = 1,...,K$ in order to obtain the system survival function. \par

In the subsequent subsections, a selection of other logical system architectures provides examples illustrating the above method.

\subsection{{\bf Bridge System}}
\label{subsec: bridge subsection}

\begin{figure}[H]
\centering
\includegraphics[scale=.25]{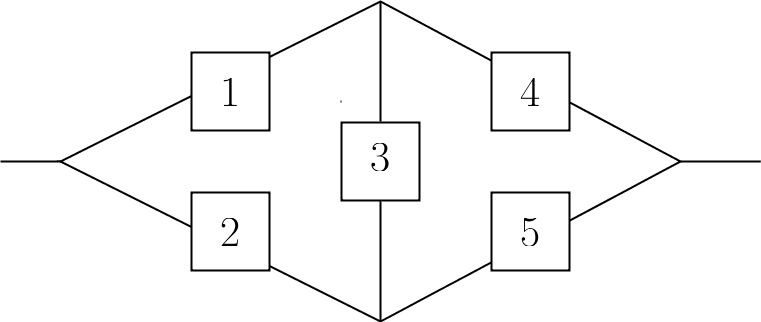}
\caption{Block Diagram of a Bridge Structure}
\label{fig: block diagram bridge structure}
\end{figure}

\hspace{.6cm}Figure~\ref{fig: block diagram bridge structure} gives the logical block diagram for a system with a bridge-style reliability. Some communications networks may use a bridge system when there are alternative ways of connecting devices such as telephones or computers. The bridge system provides many possible ways to ``complete the circuit" for relatively few components compared to a parallel system with higher reliability than a series system.\par 

 We still take each node to have its own arrival process; thus the diagram given above is logical and describes the various combinations of working components required for the system to work. It can be easily seen that the system survives past time $t$ if any one of the following sets of components all survive past $t$:
\[\{1,3,5\} \quad \{1,4\} \quad \{2,3,4\} \quad \{2,5\}\]

Thus we may give an equivalent block diagram of the bridge system using repeated components in Figure~\ref{fig: alt block diagram bridge structure}.

\begin{figure}[H]
\centering
\includegraphics[scale=.15]{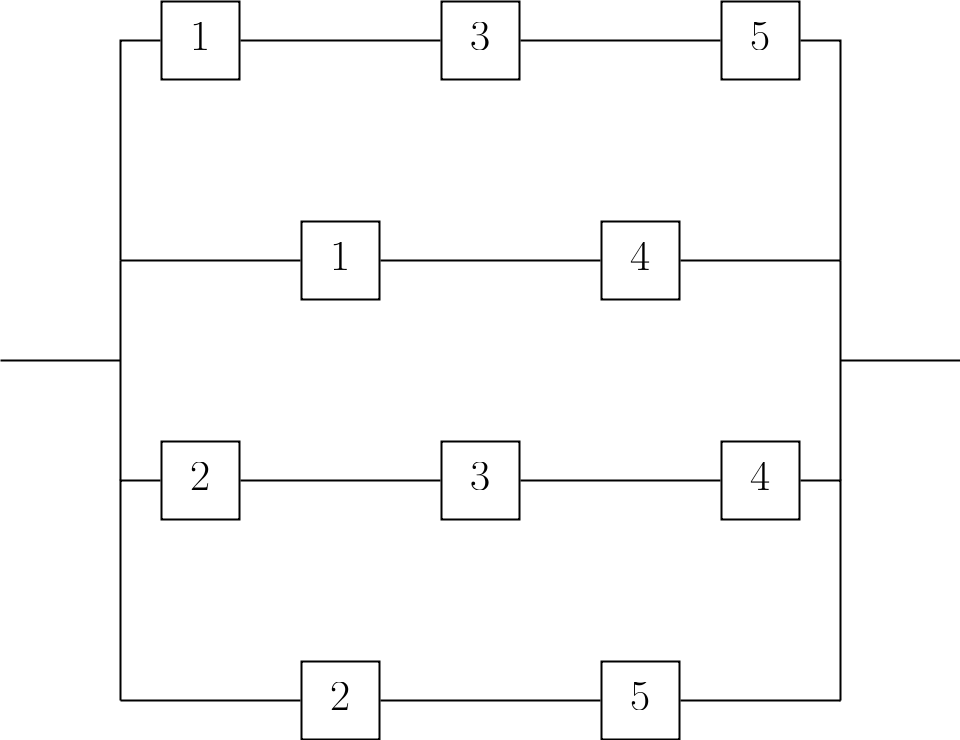}
\caption{Alternative Representation of a Bridge Structure}
\label{fig: alt block diagram bridge structure}
\end{figure}

Then the structure function may be easily derived using prior knowledge of series and parallel systems by breaking the above diagram into a parallel system of series subsystems. Therefore,
\[\phi(\mathbf{x}) = 1-(1-x_{1}x_{3}x_{5})(1-x_{1}x_{4})(1-x_{2}x_{3}x_{4})(1-x_{2}x_{5})\]

Expanding the above and replacing $x_{\ell}$ by $P(Y_{\ell}> t |N_{c}(t), \{\mcH_{j_{c}}\}, \{T_{j_{c}}\}, \{W_{j_{c}}\})$ one may derive the conditional survival function for the bridge system. Let $\mcS_{j_{c}} = \sum_{j_{c}=1}^{N_{c}(t)}\mcH_{j_{c}}\min(W_{j_{c}}, t-T_{j_{c}})$, and let 
$f^{r,q}_{\mcH}(t) = q\mcH\int_{0}^{t}e^{-q\mcH w}m_{r}(t-w)\bar{G}_{w}(w)dw$, where $r = \{1,...,5,c\}$ and $q \in \N$. Then 
\begin{align}
P(& Y_{S} > t |N_{c}(t), \{\mcH_{j_{c}}\}, \{T_{j_{c}}\}, \{W_{j_{c}}\}) \nonumber \\
		&= 		  e^{-2\mcS_{j_{c}}}\left[\bar{F}_{0_{1}}(t)\bar{F}_{0_{4}}(t)\exp\left(-E_{\mcH}\left[f_{\mcH}^{1,1}(t)+f_{\mcH}^{4,1}(t)\right]\right) \right. \nonumber \\
				&\hspace*{1in}+ \left.\bar{F}_{0_{2}}(t)\bar{F}_{0_{5}}(t)\exp\left(-E_{\mcH}\left[f_{\mcH}^{2,1}(t)+f_{\mcH}^{5,1}(t)\right]\right)\right] \nonumber \\
		& + e^{-3\mcS_{j_{c}}}\left[\bar{F}_{0_{1}}(t)\bar{F}_{0_{3}}(t)\bar{F}_{0_{5}}(t)\exp\left(-E_{\mcH}\left[f_{\mcH}^{1,1}(t)+f_{\mcH}^{3,1}(t)+f_{\mcH}^{5,1}(t)\right]\right) \right. \nonumber \\
						&\hspace*{1in}+ \left.\bar{F}_{0_{2}}(t)\bar{F}_{0_{3}}(t)\bar{F}_{0_{4}}(t)\exp\left(-E_{\mcH}\left[f_{\mcH}^{2,1}(t)+f_{\mcH}^{3,1}(t)+f_{\mcH}^{4,1}(t)\right]\right)\right]\nonumber \\
		& - e^{-4\mcS_{j_{c}}}\left[\bar{F}_{0_{1}}(t)\bar{F}_{0_{2}}(t)\bar{F}_{0_{4}}(t)\bar{F}_{0_{5}}(t)\exp\left(-E_{\mcH}\left[f_{\mcH}^{1,1}(t)+f_{\mcH}^{2,1}(t)+f_{\mcH}^{4,1}(t)+f_{\mcH}^{5,1}(t)\right]\right)\right]\nonumber \\	
		&-e^{-5\mcS_{j_{c}}}\left[\bar{F}_{0_{1}}(t)\bar{F}_{0_{3}}(t)\bar{F}_{0_{5}}(t)\exp\left(-E_{\mcH}\left[f_{\mcH}^{1,1}(t)+f_{\mcH}^{3,1}(t)+f_{\mcH}^{5,1}(t)\right]\right) \right. \nonumber \\
						&\hspace*{1in}+ \bar{F}_{0_{1}}^{2}(t)\bar{F}_{0_{3}}(t)\bar{F}_{0_{4}}(t)\bar{F}_{0_{5}}(t)\exp\left(-E_{\mcH}\left[2f_{\mcH}^{1,1}(t)+f_{\mcH}^{3,1}(t)+f_{\mcH}^{4,1}(t)+f_{\mcH}^{5,1}(t)\right]\right)\nonumber \\	 
						&\hspace*{1in}+ \bar{F}_{0_{2}}^{2}(t)\bar{F}_{0_{3}}(t)\bar{F}_{0_{4}}(t)\bar{F}_{0_{5}}(t)\exp\left(-E_{\mcH}\left[2f_{\mcH}^{2,1}(t)+f_{\mcH}^{3,1}(t)+f_{\mcH}^{4,1}(t)+f_{\mcH}^{5,1}(t)\right]\right)\nonumber \\	
						&\hspace*{1in}+ \bar{F}_{0_{4}}^{2}(t)\bar{F}_{0_{1}}(t)\bar{F}_{0_{2}}(t)\bar{F}_{0_{3}}(t)\exp\left(-E_{\mcH}\left[2f_{\mcH}^{4,1}(t)+f_{\mcH}^{1,1}(t)+f_{\mcH}^{2,1}(t)+f_{\mcH}^{3,1}(t)\right]\right)\nonumber \\
						&\hspace*{1in}+ \left.\bar{F}_{0_{5}}^{2}(t)\bar{F}_{0_{1}}(t)\bar{F}_{0_{2}}(t)\bar{F}_{0_{3}}(t)\exp\left(-E_{\mcH}\left[2f_{\mcH}^{5,1}(t)+f_{\mcH}^{1,1}(t)+f_{\mcH}^{2,1}(t)+f_{\mcH}^{3,1}(t)\right]\right)\right]\nonumber \\
		& - e^{-6\mcS_{j_{c}}}\left[\bar{F}_{0_{1}}(t)\bar{F}_{0_{2}}(t)\bar{F}_{0_{3}}^{2}(t)\bar{F}_{0_{4}}(t)\bar{F}_{0_{5}}(t)\exp\left(-E_{\mcH}\left[2f_{\mcH}^{3,1}(t)+ f_{\mcH}^{1,1}(t)+f_{\mcH}^{2,1}(t)+f_{\mcH}^{4,1}(t)+f_{\mcH}^{5,1}(t)\right]\right)\right]\nonumber \\
		& + e^{-7\mcS_{j_{c}}}\left[\bar{F}_{0_{1}}^{2}(t)\bar{F}_{0_{2}}(t)\bar{F}_{0_{3}}(t)\bar{F}_{0_{4}}(t)\bar{F}^{2}_{0_{5}}(t)\exp\left(-E_{\mcH}\left[2f_{\mcH}^{1,1}(t)+f_{\mcH}^{2,1}(t)+f_{\mcH}^{3,1}(t)+f_{\mcH}^{4,1}(t)+2f_{\mcH}^{5,1}(t)\right]\right) \right. \nonumber \\
							&\hspace*{.35in}+ \bar{F}_{0_{1}}(t)\bar{F}_{0_{2}}^{2}(t)\bar{F}_{0_{3}}(t)\bar{F}_{0_{4}}^{2}(t)\bar{F}_{0_{5}}(t)\exp\left(-E_{\mcH}\left[2f_{\mcH}^{2,1}(t)+f_{\mcH}^{1,1}(t)+f_{\mcH}^{3,1}(t)+2f_{\mcH}^{4,1}(t)+2f_{\mcH}^{5,1}(t)\right]\right)\nonumber \\	
							&\hspace*{.35in}+ \left.\bar{F}_{0_{1}}(t)\bar{F}_{0_{2}}^{2}(t)\bar{F}_{0_{3}}^{2}(t)\bar{F}_{0_{4}}(t)\bar{F}_{0_{5}}(t)\exp\left(-E_{\mcH}\left[2f_{\mcH}^{2,1}(t)+f_{\mcH}^{1,1}(t)+2f_{\mcH}^{3,1}(t)+f_{\mcH}^{4,1}(t)+2f_{\mcH}^{5,1}(t)\right]\right)\right]\nonumber \\
		& + e^{-8\mcS_{j_{c}}}\left[\bar{F}_{0_{1}}^{2}(t)\bar{F}_{0_{2}}(t)\bar{F}_{0_{3}}^{2}(t)\bar{F}_{0_{4}}^{2}(t)\bar{F}_{0_{5}}(t)\exp\left(-E_{\mcH}\left[2f_{\mcH}^{1,1}(t)+ f_{\mcH}^{2,1}(t)+2f_{\mcH}^{3,1}(t)+2f_{\mcH}^{4,1}(t)+f_{\mcH}^{5,1}(t)\right]\right)\right]\nonumber \\
		& - e^{-9\mcS_{j_{c}}}\left[\bar{F}_{0_{1}}^{2}(t)\bar{F}_{0_{2}}^{2}(t)\bar{F}_{0_{3}}^{2}(t)\bar{F}_{0_{4}}^{2}(t)\bar{F}_{0_{5}}(t)\exp\left(-E_{\mcH}\left[2f_{\mcH}^{1,1}(t)+ 2f_{\mcH}^{2,1}(t)+2f_{\mcH}^{3,1}(t)+2f_{\mcH}^{4,1}(t)+f_{\mcH}^{5,1}(t)\right]\right)\right]\nonumber \\
\label{eq: conditional survival function bridge system}
\end{align}

We may use the linearity of expectation to simply replace $e^{-q\mcS_{j_{c}}}$ by $\exp\left(E_{\mcH}\left[f_{\mcH}^{q,c}\right]\right)$ in ~\eqref{eq: conditional survival function bridge system} to obtain $S_{Y_{S}}(t)$ for the bridge system. While not completely tractable, the survival function is given in closed form and illustrates a general technique wherein we may use the expansion of a structure function in order to derive the conditional survival function for any system where all components are correlated by one process $\{N_{c}(t)\}$. This conditional survival function will always be linear in $e^{-q\mcS_{j_{c}}}$, and thus the unconditional survival function may be easily obtained using the already well-developed method. We illustrate this with another example.

\subsection{{\bf$k$-of-$n$ System}}
\label{ k of n subsection}

The $k$-of-$n$ system has $n$ nodes in parallel of which $k$ must be functioning in order for the system to remain functioning. This is a generalization of the series system ($n$-of-$n$) and the parallel system ($1$-of-$n$).  The structure function of a $k$-of-$n$ system is given by 

\begin{equation}
\phi(\mathbf{x}) = 1-\prod\limits_{\tiny\substack{ \ell_{j} = 1; \\j = 1,...,k\\ \ell_{1} < ...< \ell_{k}}}^{n}\left(1-\prod_{j=1}^{k}x_{\ell_{j}}\right)
\label{k of n structure function}
\end{equation}

By replacing $\ell_{j}$ with the appropriate conditional survival function for component $\ell$ and expanding~\eqref{k of n structure function}, we may again arrive at the system survival function conditioned upon $\{N_{c}(t): t \geq 0\}$. Because this system may also be expressed as a parallel system of series subsystems, the conditional survival function will again be linear in $e^{-q\mcS_{j_{c}}}$ and thus the linearity of expectation allows for straightfoward computation of the system survival function. 

\subsubsection{Example: 2-of-3 system}
\label{2 of 3 subsubsection}

\begin{figure}[H]
\centering
\includegraphics[scale=.25]{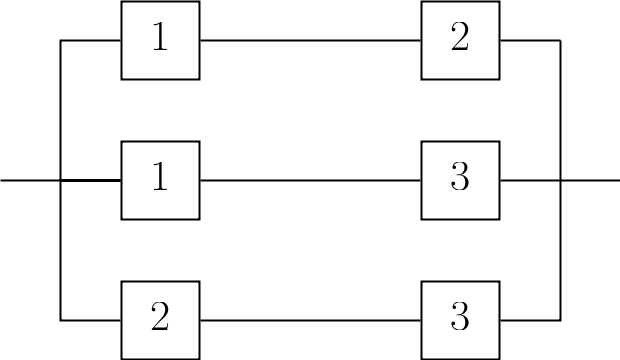}
\caption{Block Diagram of 2-of-3 System}
\label{2 of 3 structure}
\end{figure}

Figure~\ref{2 of 3 structure} gives the logical block diagram for a 2-of-3 system. Thus, using ~\eqref{k of n structure function}, the structure function for the 2-of-3 system is given by 

\begin{align}
\phi(\mathbf{x}) &= 1-(1-x_{1}x_{2})(1-x_{1}x_{3})(1-x_{2}x_{3}) \nonumber \\
						&= \sum_{i=1}^{2}\sum_{j=i+1}^{3}x_{i}x_{j} - \sum\limits_{\substack{i=1\\j,k\neq i}}^{3}x_{i}^{2}x_{j}x_{k} + \prod_{i=1}^{3}x_{i}^{2}
\end{align}

Using the technique described in Section~\ref{subsec: bridge subsection}, one may arrive at the unconditional survival function of the 2-of-3 system:

\begin{align}
S_{Y_{S}}(t) &= \exp\left(-E_{\mcH}[f_{\mcH}^{c,2}(t)]\right)\left[\sum_{i=1}^{2}\sum_{j=i+1}^{3}\bar{F}_{0_{i}}(t)\bar{F}_{0_{j}}(t)\exp\left(-E_{\mcH}[f_{\mcH}^{i,1}(t)+f_{\mcH}^{j,1}(t)]\right)\right] \nonumber \\
	&\quad \exp\left(-E_{\mcH}[f_{\mcH}^{c,4}(t)]\right)\left[\sum\limits_{\substack{i=1\\j,k\neq i}}^{3}\bar{F}_{0_{i}}^{2}(t)\bar{F}_{0_{j}}(t)\bar{F}_{0_{k}}(t)\exp\left(-E_{\mcH}\left[2f_{\mcH}^{i,1}(t)+f_{\mcH}^{j,1}(t)+f_{\mcH}^{k,1}(t)\right]\right)\right] \nonumber \\
	&\quad +\exp\left(-E_{\mcH}[f_{\mcH}^{c,6}(t)]\right)\left(\prod_{\ell=1}^{3}\bar{F}_{0_{\ell}}^{2}(t)\right)\exp\left(-E_{\mcH}\left[2\sum_{\ell=1}^{3}f_{\mcH}^{\ell,1}(t)\right]\right) 
\end{align}

\section{\large{ \bf CONCLUSION}}
\label{sec: conclusion}

\hspace*{.6cm} In~\cite{traylor}, Traylor derived the survival function for a single server under workload where the arrivals may be described by a NHPP and the workload stresses are random. We created a logical system of such servers correlated by a nonhomogenous ``correlator" Poisson process; thus each node is conditionally independent. The conditional survival function for each server (or node) is found to be in one-to-one correspondence with the binary state (operational or failed) $x_{\ell}$, and thus the system structure function is isomorphic to the system's conditional survival function. This gives a straightforward method to compute the conditional survival function of the correlated system. All structure functions may be derived using series and parallel structure functions, and both series and parallel structure functions are linear in $e^{-q\mcS_{j_{c}}}$, where $\mcS_{j_{c}} =   \sum_{j_{c}=1}^{N_{c}(t)}\mcH_{j_{c}}\min(W_{j_{c}}, t-T_{j_{c}})$. The unconditional survival function is calculated by taking the expectation over $N_{c}(t)$, and thus the linearity of expectation allows for a straightfoward computation of the system survival function. \par 
This method allows for the modeling of more complex logical system architectures without simplification of the arrival process or the assumption of node independence.  
\appendix
 \section{\large{ \bf APPENDIX}}
 \label{appendix}
 
 \begin{lemma}
 Let $\{N_{i}(t) : t \geq 0\}_{i=1}^{n}$ be independent nonhomogeneous Poisson processes with intensities $\lambda_{i}(t), i =1,...,n$. Let $N(t) = \sum_{i=1}^{n}N_{i}(t)$ is a nonhomogeneous Poisson process with intensity $\lambda(t) = \sum_{i=1}^{n}\lambda_{i}(t)$
 \label{lemma: sum of NHPP is a NHPP}
 \end{lemma}
 
 \begin{proof}
 The proof will proceed via induction. As a base case, let $n=2$. It suffices to show that $N(0) = 0$, and 
 \[P(N(t+s) - N(t) = n) = \dfrac{\exp\left(-(m(t+s) - m(t))\right)[m(t+s)-m(t)]^{n}}{n!}\]
 where $m(t) = m_{1}(t)+m_{2}(t)$ and $m_{i}(t) = \int_{0}^{t}\lambda_{i}(s)ds$.
 Clearly, $N(0) = N_{1}(0) + N_{2}(0) = 0+0 = 0$. Now, since $\{N_{1}(t)\}, \{N_{2}(t)\}$ are independent, we may find the distribution of $\{N(t)\}$ via convolution. Thus,
 	\begin{align}
 	P(N(t+s)&-N(t)=n) \nonumber \\
 	&= P((N_{1}+N_{2})(t+s) - (N_{1}+N_{2})(t)) = n) \nonumber \\
 							  &= P(N_{1}(t+s) + N_{2}(t+s) - N_{1}(t) -N_{2}(t) = n) \nonumber \\
 							  &= \sum_{x=0}^{n}P([N_{1}(t+s) - N_{1}(t) = x] \cap [N_{2}(t+s) - N_{2}(t)= n-x]) \nonumber \\
 							  &= \sum_{x=0}^{n}P(N_{1}(t+s)-N_{1}(t) = x)P(N_{2}(t+s)-N_{2}(t) = n-x) \nonumber \\
 							  &= \sum_{x=0}^{n}\frac{(m_{1}(t+s) - m_{1}(t))^{x}e^{-(m_{1}(t+s)-m_{1}(t))}}{x!}
 							  		\frac{(m_{2}(t-s) - m_{2}(t))^{n-x}e^{-(m_{2}(t+s)-m_{2}(t))}}{(n-x)!} \nonumber \\
 							  &= \frac{1}{n!}e^{-(m_{1}+m_{2})(t+s) - (m_{1}+m_{2})(t)}\sum_{x=0}^{n}\frac{n!}{x!(n-x)!}(m_{1}(t+s) - m_{1}(t))^{x}(m_{2}(t+s) - m_{2}(t))^{n-x}\nonumber \\
 							  &= \frac{1}{n!}e^{-(m_{1}+m_{2})(t+s) - (m_{1}+m_{2})(t)}(m_{1}(t+s)-m_{1}(t) + m_{2}(t+s)-m_{2}(t))^{n} \nonumber \\
 							  &= \frac{e^{-(m_{1}+m_{2})(t+s) - (m_{1}+m_{2})(t)}}{n!}((m_{1}+m_{2})(t+s) - (m_{1}+m_{2})(t))^{n} \nonumber \\
 							  &= \frac{e^{-(\lambda_{1}+\lambda_{2})(s)}(\lambda_{1}+\lambda_{2})(s))^{n}}{n!} \nonumber \\
 	\end{align}
 Now, assume that $N_{\kappa}(t) = \sum_{i=1}^{k}N_{i}(t)$ is a NHPP with intensity $\lambda(t) = \sum_{i=1}^{k}\lambda_{i}(t)$. Then let $N(t) = N_{\kappa}(t) + N_{k+1}(t)$, where $\{N_{k+1}(t)\}$ is a NHPP with intensity $\lambda_{k+1}(t)$. Then using the same procedure as above, we see that 
 \[P(N(t+s)-N(t) = n) = \frac{e^{-\left(\sum_{i=1}^{k+1}\lambda_{i}\right)(s)}(\left(\sum_{i=1}^{k+1}\lambda_{i}\right)(s))^{n}}{n!}\]
 and thus the sum of nonhomogeneous Poisson processes remains a NHPP.
 \end{proof}
 
  \begin{lemma}
  Under the condition that $N_{c}(t) =n_{c}, \mcH_{j_{c}} = \eta_{i_{j_{c}}}, i \in \{1,...,m\}, j_{c} = 1,...,n_{c}$, we have that
  \begin{align}
  P\left(\right.Y_{\ell,c} > t &\left| N_{c}(t), \{T_{j_{c}}\}_{j_{c}=1}^{N_{c}(t)}, \{W_{j_{c}}\}_{j_{c}=1}^{N_{c}(t)}, \{\mcH_{j_{c}}\}_{j_{c}=1}^{N_{c}(t)}  \right.\left.\right) \nonumber \\
  	&= \bar{F}_{0_{\ell,c}}\exp\left(- \sum_{j_{c}=1}^{N_{c}(t)}\mcH_{j_{c}}\min(W_{j_{c}}, t-T_{j_{c}})\right)\exp\left(-E_{\mcH}\left[\mcH\int_{0}^{t}\exp(-\mcH w)m_{\ell}(t-w)\bar{G}_{W}(w)dw\right]\right) \nonumber
  \end{align}
  where $m_{\ell}(x) = \int_{0}^{x}\lambda_{\ell}(s)ds$.
  \label{lemma: conditional survival function for a node}
  \end{lemma}
  
  \begin{proof}
  As before (\cite{traylor}), we have that 
  \begin{align}
 P\left(\right.Y_{\ell} > t &\left| N_{c}(t), \{T_{j_{c}}\}_{j_{c}=1}^{N_{c}(t)}, \{W_{j_{c}}\}_{j_{c}=1}^{N_{c}(t)}, \{\mcH_{j_{c}}\}_{j_{c}=1}^{N_{c}(t)}, N_{\ell}(t), \{T_{j_{\ell}}\}_{j_{\ell}=1}^{N_{\ell}(t)}, \{W_{j_{\ell}}\}_{j_{\ell}=1}^{N_{\ell}(t)}, \{\mcH_{j_{\ell}}\}_{j_{\ell}=1}^{N_{\ell}(t)}  \right.\left.\right) \nonumber \\
 	&= \exp\left(-\int_{0}^{t}\mcB_{\ell}(t)\right) \nonumber \\
 	&= \bar{F}_{0_{1}}(t)\exp\left(-\sum_{j_{\ell}=1}^{N_{\ell}(t)}\mcH_{j_{\ell}}\min(W_{j_{\ell}}, t-T_{j_{\ell}}) - \sum_{j_{c}=1}^{N_{c}(t)}\mcH_{j_{c}}\min(W_{j_{c}}, t-T_{j_{c}})\right) \nonumber \\
 	&= \bar{F}_{0_{1}}(t)\exp\left(-\sum_{j_{\ell}=1}^{N_{\ell}(t)}\mcH_{j_{\ell}}\min(W_{j_{\ell}}, t-T_{j_{\ell}}) \right)\exp\left(- \sum_{j_{c}=1}^{N_{c}(t)}\mcH_{j_{c}}\min(W_{j_{c}}, t-T_{j_{c}})\right)
  \end{align}
  
  Now, 
  \begin{align}
   P\left(\right.Y_{\ell,c} > t &\left| N_{c}(t), \{T_{j_{c}}\}_{j_{c}=1}^{N_{c}(t)}, \{W_{j_{c}}\}_{j_{c}=1}^{N_{c}(t)}, \{\mcH_{j_{c}}\}_{j_{c}=1}^{N_{c}(t)}  \right.\left.\right) \nonumber \\ 
   &= E_{N_{\ell},\{\mcH_{j_{\ell}}\}}\left[\bar{F}_{0_{\ell,c}}(t)\exp\left(-\sum_{j_{\ell}=1}^{N_{\ell}(t)}\mcH_{j_{\ell}}\min(W_{j_{\ell}}, t-T_{j_{\ell}}) \right)\exp\left(- \sum_{j_{c}=1}^{N_{c}(t)}\mcH_{j_{c}}\min(W_{j_{c}}, t-T_{j_{c}})\right)\right] \nonumber \\
   &= \bar{F}_{0_{\ell,c}}\exp\left(- \sum_{j_{c}=1}^{N_{c}(t)}\mcH_{j_{c}}\min(W_{j_{c}}, t-T_{j_{c}})\right)E_{N_{\ell},\{\mcH_{j_{\ell}}\}}\left[\exp\left(-\sum_{j_{\ell}=1}^{N_{\ell}(t)}\mcH_{j_{\ell}}\min(W_{j_{\ell}}, t-T_{j_{\ell}}) \right)\right] \nonumber \\
   \intertext{But this case reduces to the previous RSBR case, and hence we have }
   &= \bar{F}_{0_{\ell,c}}\exp\left(- \sum_{j_{c}=1}^{N_{c}(t)}\mcH_{j_{c}}\min(W_{j_{c}}, t-T_{j_{c}})\right)\exp\left(-E_{\mcH}\left[\mcH\int_{0}^{t}\exp(-\mcH w)m(t-w)\bar{G}_{W}(w)dw\right]\right) \nonumber 
  \end{align}
  \end{proof}
  
  \bibliographystyle{acm}
  \bibliography{NHPP_corr}

 \end{document}